\documentclass[11pt,oneside]{article}
\input{epsf.sty}
\usepackage{exscale}
\usepackage{relsize}
\usepackage{amsthm}
\usepackage{amsmath}
\usepackage{latexsym,bm}
\usepackage{floatflt}
\usepackage{fancyhdr}
\usepackage{amsfonts}
\usepackage{mathrsfs}
\usepackage{amssymb}
\usepackage{mathrsfs}
\usepackage{indentfirst}
\usepackage{color}
\usepackage{ifpdf}
\usepackage{bm}
\allowdisplaybreaks

\allowdisplaybreaks

\newtheorem{theorem}{\indent{Theorem}}[section]
\newtheorem{proposition}{\indent{Proposition}}[section]
\newtheorem{lemma}{\indent{Lemma}}[section]

\newtheorem{definition}{\indent{Definition}}[section]
\newtheorem{remark}{\indent{Remark}}[section]
\newtheorem{example}{\indent{Example}}[section]

\begin{document}

\pagenumbering{arabic}
\title
{Center-focus determination and limit cycles bifurcation for $p:q$ homogeneous weight singular point
   \thanks{This research was  supported by the National Natural Science Foundation of China (11371373) }}
\author{{Tao Liu}\\
{\small \it School of Mathematics,  Central South University, Changsha, Hunan, 410083, P.R. China}\\
{Feng Li}\\
{\small \it School of Science, Linyi University,Linyi ,Shandong, 276005, P.R. China}\\
{Yirong Liu}\\
{\small \it School of Mathematics,  Central South University,Changsha, Hunan, 410083, P.R. China}\\
{Shimin Li}\\
{\small \it School of Mathematics and Statistics, Guangdong University of Finance and Economics, }\\
{\small \it  Guangzhou, Guangdong, 510320, P.R. China}\\
}
\date{}
\maketitle \pagestyle{myheadings} \markboth{ Tao Liu, Feng Li, Yirong  Liu and Shimin Li
}{Center-focus determination and limit cycle bifurcation for $p:q$ homogeneous weight singular point} \noindent

\begin{abstract}
The quasi-homogeneous (and in general non-homogeneous) polynomial differential systems have
been studied from many different points of view. In this paper, Center-focus determination and limit cycles bifurcation for $p:q$ homogeneous weight singular point are investigated.
Some prosperities of Successive function and focus values are discussed, furthermore, the method of computing focal values is given. As an example, Center-focus determination and limit cycle bifurcation for $2:3$ homogeneous weight singular point are studied, three or five limit cycles in the neighborhood of origin can be obtained by different perturbations.

\end{abstract}



\section{Introduction}
It is well known that center-focus determination is difficult and important in qualitative theory of planar system.  It is far from being solved although there were many results for elementary singular point and nilpotent singular point \cite{liu-2014}. Hopf bifurcation has also been investigated intensively because it is closely related to center-focus problem and the 16th problem of Hilbert. For planar ordinary differential equations, there were many good results for planar systems. For example,  one of the best-known results was $M(2)=3$ \cite{Bautin1952} for a planar system with an elementary critical point. Here, $M(n)$ denotes the maximal number of
small-amplitude limit cycles around a singular point with $n$ being the
degree of polynomials in the vector field. When $n = 3$,  the authors constructed two different cubic systems to show there exist $9$ limit cycles for cubic systems in \cite{Zoladek1995} and \cite{Christopher2005}.  Recently, Yu and Tian showed that there could be twelve limit cycles around a singular point in a planar cubic-degree polynomial system\cite{Yu2014}.
But for a system  with a degenerate critical points it is still a hard work to solve its center problem and to determine the number of limit cycles.
When a critical point is  degenerate, its center problem has also been investigated by many authors, see \cite{41,42,44,45,38,39}. There were also many results about the bifurcation of limit cycles \cite{43,46,47,48}, for more detail, see \cite{Liu2010,34}. A special system with total degenerate critical point was investigated by Liu etc. in \cite{liu2001}.

Some special systems have also been investigated. The homogeneous polynomial differential systems have been studied by several authors. Thus, the
quadratic homogeneous ones by \cite{12,18,24,25,26,27,29}; the cubic homogeneous ones by \cite{9}; the homogeneous systems of arbitrary degree by \cite{7,9,10,19}, and others. In these previous papers is described
an algorithm for studying the phase portraits of homogeneous polynomial vector fields for all degree, the classification of all phase portraits of homogeneous polynomial vector fields of degree 2
and 3, the algebraic classifications of homogeneous polynomial vector fields and the characterization
of structurally stable homogeneous polynomial vector fields.

The quasi-homogeneous (and in general non-homogeneous) polynomial differential systems have
been studied from many different points of view, mainly for their integrability \cite{3,4,14,15,16,17,22}, for their
rational integrability \cite{5,30,31,32}, for their polynomial integrability \cite{8,23,28}, for their centers \cite{1,2,20},
for their normal forms \cite{6}, for their limit cycles \cite{21}, ... . But up to now there was not an algorithm
for constructing all the quasi-homogeneous polynomial differential systems of a given degree. Han and xiong  classified all centers of a class of quasi-homogeneous polynomial differential systems of degree 5 in \cite {han2015}. The same authors  investigated a class of quasi-homogeneous polynomial systems with a given weight degree in \cite{han2016}.  The cyclicity and center problems are studied for some subfamilies of semi-quasihomogeneous polynomial systems by Zhao in \cite{zhao2013}.

In this paper, center-focus determination and limit cycle bifurcation for $p:q$ homogeneous weight singular point will be investigated. In section two, homogeneous weight system and generalized polar coordinate are given; Some prosperities of successive function and focus values are discussed in section three; As application, center-focus determination and limit cycle bifurcation for $2:3$ homogeneous weight singular point are investigated
in section four.

\section{Homogeneous weight system and generalized Polar Coordinate}

 \setcounter{equation}{0}
 In this section, some necessary definitions are given.
\begin{definition}\label{d1.1}
If there exist positive integer $p,q,m$, which satisfy
\begin{equation}\label{e1.1}
    F(\lambda^px,\lambda^qy)\equiv\lambda^mF(x,y),
\end{equation}
then $F(x,y)$ is called to be $m-$order homogeneous weight function of $x,y$  with  weight $p,q$.
\end{definition}

It is easy to testify that if $p,q,m$ are positive integers and $F(x,y)$ is a $m-$order homogeneous weight polynomial function of $x,y$  with  weight $p,q$, then $F(x,y)$ can be written as
\begin{equation}\label{e1.2}
    F(x,y)=\sum_{kp+jq=m}C_{kj}x^ky^j.
\end{equation}

Considering the following system
 \begin{equation}\label{e1.3}
    \frac{dx}{dt}=-\lambda_1y^{2p-1}+\sum_{k+j=2}^{\infty}a_{kj}x^ky^j,\ \
     \frac{dy}{dt}=\lambda_2x^{2q-1}+\sum_{k+j=2}^{\infty}b_{kj}x^ky^j,
 \end{equation}
where  $p,q$ are positive integer,\ $a_{0,2p-1}=b_{2q-1,0}=0$,\ $\lambda_1>0,\lambda_2>0$.\ Suppose functions of right hand of system \eqref{e1.3} are power series of $x,y$ with non-zero radius convergence.\par
Without loss of generality, let
 \begin{equation}\label{e1.4}
    \lambda_1=p,\ \lambda_2=q,
 \end{equation}
otherwise, let
 \begin{equation}\label{e1.5}
    x=\sqrt[2q]{\frac{q}{\lambda_2}}\ u,\ y=\sqrt[2p]{\frac{p}{\lambda_1}}\ v,\ \
    \frac{dt}{d\tau}= \sqrt[2p]{\frac{p}{\lambda_1}}\ \sqrt[2q]{\frac{q}{\lambda_2}}.
 \end{equation}
When \eqref{e1.4} holds, system \eqref{e1.3} could be rewritten as
 \begin{equation}\label{e1.6}
    \frac{dx}{dt}=-py^{2p-1}+\sum_{k+j=2}^{\infty}a_{kj}x^ky^j,\ \
     \frac{dy}{dt}=qx^{2q-1}+\sum_{k+j=2}^{\infty}b_{kj}x^ky^j,
 \end{equation}

 The origin of system \eqref{e1.6} is an elementary singular point when $p=q=1$ and a nilpotent singular points when $p=1,\ q>1$(or $p>1$, $q=1$), there are many differential topological constructions of phase curves in the neighborhood of origin of system \eqref{e1.6}, see \cite{Zhang-1985}. When $p>1,q>1$, topological constructions of phase curves in the neighborhood of origin of system \eqref{e1.6} has not been investigated completely. \par
 In this paper, we do not consider the topological constructions of phase curves in the neighborhood of origin of system \eqref{e1.6}.
For sufficiently small $h>0$, the solution of system \eqref{e1.6} which satisfy initial condition $x|_{t=0}=h^p,\ y|_{t=0}=0$ goes around the origin at the neighborhood of $x^{2q}+y^{2p}=h^{2pq}$,
then  phase curves in the neighborhood of origin of system \eqref{e1.6} can be studied by transformations
\begin{equation}\label{e1.7}
    x=r^p\cos\theta,\ \ y=r^q\sin\theta
\end{equation}
So the system could be written as
\begin{equation}\label{e1.8}
\begin{split}
    &\frac{dx}{dt}=-py^{2p-1}+\sum_{kp+jq> \atop
                                     (2p-1)q}
    ^{\infty}a_{kj}x^ky^j=\mathcal{X}(x,y),\\
    & \frac{dy}{dt}=qx^{2q-1}+\sum_{kp+jq>\atop
                                     (2q-1)p}^{\infty}b_{kj}x^ky^j=\mathcal{Y}(x,y).
     \end{split}
 \end{equation}

\begin{definition}\label{d1.2}
system \eqref{e1.8} is called to be $p:q$ homogeneous weight system, and the origin of system \eqref{e1.8} is called to be homogeneous weight focus (or center) with weight $p:q$.
\end{definition}

 \begin{example}\label{ex1.2}
 Homogeneous weight system with weight $2:3$ can be written as
  \begin{equation}\label{e1.9}
  \begin{split}
&\frac{dx}{dt} =-2y^3+\sum_{2k+3j=10}^{\infty}a_{kj}x^ky^j\\
= &-2y^3+(a_{22}x^2y^2+a_{13}xy^3+a_{04}y^4)+\sum_{k+j=5}^{\infty}a_{kj}x^ky^j,\\
   &\frac{dy}{dt} =3x^5+\sum_{2k+3j=11}^{\infty}b_{kj}x^ky^j=3x^5+(b_{13}xy^3+b_{04}y^4)\\
    +&(b_{41}x^4y+b_{32}x^3y^2+b_{23}x^2y^3+b_{14}xy^4+b_{05}y^5)+\sum_{k+j=6}^{\infty}b_{kj}x^ky^j.
     \end{split}
  \end{equation}
 \end{example}

Now, the functions of right hand of system \eqref{e1.8} can be written as a homogeneous weight polynomial power seriesㄩ
\begin{equation}\label{e1.10}
    \begin{split}
    &\frac{dx}{dt}=-py^{2p-1}+\sum_{m=2pq-q+1}^{\infty}\mathcal{X}_m(x,y)=\mathcal{X}(x,y),\\
    &\frac{dy}{dt}=qx^{2q-1}+\sum_{m=2pq-p+1}^{\infty}\mathcal{Y}_m(x,y)=\mathcal{Y}(x,y),
    \end{split}
\end{equation}
 where
 \begin{equation}\label{e1.11}
   \mathcal{X}_m(x,y)=\sum_{kp+jq=m}a_{kj}x^ky^j,\ \   \mathcal{Y}_m(x,y)=\sum_{kp+jq=m}b_{kj}x^ky^j
 \end{equation}
are $m-$order homogeneous weight polynomial of $x,y$  with  weight $p,q$ which satisfy
 \begin{equation}\label{e1.12}
\begin{split}
    &\mathcal{X}_m(r^p\cos\theta,r^q\sin\theta)=r^m\mathcal{X}_m(\cos\theta,\sin\theta),\\
    &\mathcal{Y}_m(r^p\cos\theta,r^q\sin\theta)=r^m\mathcal{Y}_m(\cos\theta,\sin\theta).
    \end{split}
\end{equation}
Taking the derivative of \eqref{e1.7} with $t$, we have
\begin{equation}\label{e1.13}
    \begin{split}
    &\mathcal{X}=p\ r^{p-1}\cos\theta \frac{dr}{dt}-r^p\sin\theta \frac{d\theta}{dt}\\
    &\mathcal{Y}=q\ r^{q-1}\sin\theta \frac{dr}{dt}+r^q\cos\theta \frac{d\theta}{dt}.
    \end{split}
\end{equation}
\eqref{e1.13} yields that
\begin{equation}\label{e1.14}
    \begin{split}
    &\frac{dr}{dt}=r\ \frac{r^q\cos\theta \mathcal{X}+r^p\sin\theta \mathcal{Y}}{r^{p+q}(p\cos^2\theta+q\sin^2\theta)}
    =\frac{r^{2pq+1}}{{r^{p+q}(p\cos^2\theta+q\sin^2\theta)}}\sum\limits_{k=0}^{\infty}R_k(\theta)r^k ,\\
    &\frac{d\theta}{dt}=\frac{-qr^q\sin\theta \mathcal{X}+pr^p\cos\theta \mathcal{Y}}{r^{p+q}(p\cos^2\theta+q\sin^2\theta)}
      = \frac{r^{2pq}}{{r^{p+q}(p\cos^2\theta+q\sin^2\theta)}}\sum\limits_{k=0}^{\infty}Q_k(\theta)r^k.
    \end{split}
\end{equation}
where
\begin{equation}\label{e1.15}
\begin{split}
   & R_k(\theta)=\cos\theta\ \mathcal{X}_{2pq-q+k}(\cos\theta,\sin\theta)+\sin\theta\ \mathcal{Y}_{2pq-p+k}(\cos\theta,\sin\theta),\\
   &Q_k(\theta)=-q\sin\theta\ \mathcal{X}_{2pq-q+k}(\cos\theta,\sin\theta)+p\cos\theta\ \mathcal{Y}_{2pq-p+k}(\cos\theta,\sin\theta),
   \end{split}
\end{equation}
Especially,
\begin{equation}\label{e1.16}
\begin{split}
     &R_0(\theta)=\cos\theta\sin\theta\left(q\cos^{2q-2}\theta-p\sin^{2p-2}\theta\right),\\
     &Q_0(\theta)=pq\left(\cos^{2q}\theta+\sin^{2p}\theta\right)>0.
     \end{split}
\end{equation}

\eqref{e1.14} yields that by transformation \eqref{e1.7} system \eqref{e1.10} could be rewritten as the following equation:
\begin{equation}\label{e1.17}
    \frac{dr}{d\theta}=r\ \frac{\sum\limits_{k=0}^{\infty}R_k(\theta)r^k}{\sum\limits_{k=0}^{\infty}Q_k(\theta)r^k}
    =\frac{R_0(\theta)}{Q_0(\theta)}\ r+o(r).
\end{equation}
because $Q_0(\theta)\neq0$, for sufficiently small $h$, the solution of system \eqref{e1.17} which satisfy initial condition
\begin{equation}\label{e1.18}
    r|_{\theta=0}=h
\end{equation}
is a power series of $h$ with non-zero radius convergence when $|\theta|<4\pi$. Let
\begin{equation}\label{e1.19}
    r=\tilde{r}(\theta,h)=\sum_{k=1}^{\infty}\nu_k(\theta)h^k.
\end{equation}
where
\begin{equation}\label{e1.20}
    \nu_1(0)=1,\ \ \nu_k(0)=0,\ k=2,3,\cdots.
\end{equation}

Furthermore
\begin{equation}\label{e1.21}
    \nu_1(\theta)=\exp\int_0^{\theta}\frac{R_0(\theta)}{Q_0(\theta)}d\theta
    =\left(\cos^{2q}\theta+\sin^{2p}\theta\right)^{\frac{-1}{2pq}}.
\end{equation}
So the origin of system \eqref{e1.10} is a focus or center, the Poincar\'{e} successive function in the neighborhood of the origin can be written as
\begin{equation}\label{e1.22}
    \triangle(h)=\tilde{r}(2\pi,h)-h=\sum_{k=2}^{\infty}\nu_k(2\pi)h^k.
\end{equation}

\begin{remark}\label{r1.1}
If $p ,q$ are not relatively prime, namely, there exists  positive constants $d\ ,p^*\ ,q^*$ which satisfy
\begin{equation}\label{1.23}
    p=dp^*,\ q=dq^*,
\end{equation}
where $d>1$.
Then the transformation \eqref{e1.7} is equivalent to
\begin{equation}\label{e1.24}
    x=(r^*)^{p^*}\cos\theta,\ \ y=(r^*)^{q^*}\sin\theta
\end{equation}
where
\begin{equation}\label{e1.25}
    r^*=r^d.
\end{equation}
and the system \eqref{e1.17} could be changed into equation of polar coordinates by transformation \eqref{e1.25}.
\end{remark}

\section{Some prosperities of Successive function and focus values}
 \setcounter{equation}{0}

Because the functions of right hand of system \eqref{e1.17} is periodic function of $\theta$ with period $2\pi$,
\begin{proposition}\label{p2.1}
For sufficiently small constant $h$, when $|\theta|<4\pi$, we have
\begin{equation}\label{e2.1}
    \tilde{r}(\theta+2\pi,h)=\tilde{r}(\theta,\tilde{r}(2\pi,h)).
\end{equation}
\end{proposition}

\begin{proposition}\label{p2.2}
If $p,q$ are prime numbers, for sufficiently small constant $h$, when $|\theta|<4\pi$, we have
\begin{equation}\label{e2.2}
    -\tilde{r}(\theta+\pi,h)=\tilde{r}(\theta,-\tilde{r}(\pi,h)).
\end{equation}
\end{proposition}
\begin{proof}
Because $p,q$ are prime numbers, the transformation \eqref{e1.7} is equivalent to
\begin{equation}\label{e2.3}
    x=(-r)^p\cos(\theta+\pi),\ \ y=(-r)^q\sin(\theta+\pi).
\end{equation}
It is easy to testify that system \eqref{e1.17} keep formally unchanged by transformation $r\rightarrow -r,\ \theta\rightarrow\theta+\pi$,
so $r=-\tilde{r}(\theta+\pi,h)$ is a solution of \eqref{e1.17} which satisfy initial condition $r|_{\theta=0}=-\tilde{r}(\pi,h)$.\
On the other hand, $r=\tilde{r}(\theta,-\tilde{r}(\pi,h))$ is another solution of \eqref{e1.17} which satisfy the same initial condition.
So we can get \eqref{e2.2} easily by uniqueness theorem for the solution.
\end{proof}

\begin{proposition}\label{p2.3}
If $p, q$ are even number, then $\tilde{r}(\theta,h)$ is an odd function of  $h$.
\end{proposition}
\begin{proof}
 If $p, q$ are even number, then $X(r^p\cos\theta,r^q\sin\theta)$ and $Y(r^p\cos\theta,r^q\sin\theta)$ are even functions of $r$, so the functions of right han
 of system \eqref{e1.17} is an odd function of $r$, which shows that Proposition \ref{p2.2} holds.
\end{proof}

\begin{proposition}\label{p2.4}
 If $p$ is an odd number, $q$ is an even number, for sufficiently small constant $h$, when $|\theta|<4\pi$, we have
 \begin{equation}\label{e2.4}
    -\tilde{r}(\pi-\theta,h)=\tilde{r}(\theta,-\tilde{r}(\pi,h)).
 \end{equation}
 \end{proposition}
\begin{proof}
 Because $p$ is an odd number, $q$ is an even number, the transformation \eqref{e1.7} is equivalent to
\begin{equation}\label{e2.5}
    x=(-r)^p\cos(\pi-\theta),\ \ y=(-r)^q\sin(\pi-\theta).
\end{equation}
It is easy to testify that system \eqref{e1.17} keep formally unchanged by transformation $r\rightarrow -r,\ \theta\rightarrow\pi-\theta$, so
 $r=-\tilde{r}(\pi-\theta,h)$ is a solution of \eqref{e1.17} which satisfy initial condition $r|_{\theta=0}=-\tilde{r}(\pi,h)$.\
 on the other hand, $r=\tilde{r}(\theta,-\tilde{r}(\pi,h))$ is another solution of  \eqref{e1.17} which satisfy the same initial condition.
So we can get \eqref{e2.5} easily by uniqueness theorem for the solution.
\end{proof}

\begin{proposition}\label{p2.5}
 If $p$ is a even number, $q$ is an odd number, for sufficiently small constant $h$, when $|\theta|<4\pi$, we have
 \begin{equation}\label{e2.6}
    -\tilde{r}(\pi-\theta,h)=\tilde{r}(\theta,-\tilde{r}(\pi,h)).
 \end{equation}
 \end{proposition}
\begin{proof}
 Because $p$ is an even number, $q$ is an odd number, the transformation \eqref{e1.7} is equivalent to
\begin{equation}\label{e2.7}
    x=(-r)^p\cos(2\pi-\theta),\ \ y=(-r)^q\sin(2\pi-\theta).
\end{equation}
It is easy to testify that system \eqref{e1.17} keep formally unchanged by transformation $r\rightarrow -r,\ \theta\rightarrow2\pi-\theta$, so
 $r=-\tilde{r}(2\pi-\theta,h)$ is a solution of \eqref{e1.17} which satisfy initial condition $r|_{\theta=0}=-\tilde{r}(2\pi,h)$.\
 on the other hand, $r=\tilde{r}(\theta,-\tilde{r}(2\pi,h))$  is another solution of  \eqref{e1.17} which satisfy the same initial condition.
So we can get \eqref{e2.6} easily by uniqueness theorem for the solution.
\end{proof}

Suppose $f(h)=\sum\limits_{k=0}^{\infty}c_kh^k$ is a form series of $h$,  we denote the coefficient of $h^m$ of $f(h)$ by $\Big[f(h)\Big]_m$, namely,
$\Big[f(h)\Big]_m=c_m$. Then

\begin{lemma}\label{l2.1}
Suppose $m$ is a positive integer and greater than 1, if $\nu_k(2\pi)=0$ when $1<k<m$, namely
 \begin{equation}\label{e2.8}
    \tilde{r}(2\pi,h)=h+\nu_m(2\pi)h^m+o(h^m),
 \end{equation}
 then
  \begin{equation}\label{e2.9}
      \left[\tilde{r}^k(2\pi,h)\right]_m=\left\{\begin{array}{cl}
                                                  \nu_m(2\pi), &if\  k=1, \\
                                                  0, &if\  1<k<m ,\\
                                                  1, &if\  k=m ,\\
                                                  0, &if\  k>m.
                                                \end{array}\right.
 \end{equation}
 \end{lemma}

\begin{theorem}\label{t2.1}
If $p+q$ is an even number, when $k>1$,  the first subscript satisfying $\nu_k(2\pi)\neq 0$ in $\{\nu_k(2\pi)\}$ is an odd number.
\end{theorem}

\begin{proof}
 If $p$ and $q$ are even numbers,  $\tilde{r}(\theta,h)$ is an odd function of $h$ by Proposition\ref{p2.3}, obviously, Theorem \ref{t2.1} holds\par
Next we suppose $p$ and $q$ are odd numbers. Let $\theta=\pi$ in $\eqref{e2.1}$ and $\theta=2\pi$ in \eqref{e2.2},then
  \begin{equation}\label{e2.10}
  \begin{split}
   \tilde{ r}(3\pi,h)&=\tilde{ r}(\pi,\tilde{ r}(2\pi,h)),\\
   -\tilde{ r}(3\pi,h)&=\tilde{ r}(2\pi,-\tilde{ r}(\pi,h)).
   \end{split}
  \end{equation}
From \eqref{e2.10}, we have
 \begin{equation}\label{e2.11}
    \tilde{ r}(\pi,\tilde{ r}(2\pi,h))+\tilde{ r}(2\pi,-\tilde{ r}(\pi,h))=0.
 \end{equation}
  Namely
  \begin{equation}\label{e2.12}
   \sum_{k=1}^{\infty}\nu_k(\pi)\tilde{ r}^k(2\pi,h)+\sum_{k=1}^{\infty}(-1)^k\nu_k(2\pi)\tilde{ r}^k(\pi,h)=0.
 \end{equation}
 For any positive integer $m$, \eqref{e2.12} yields that
 \begin{equation}\label{e2.13}
    \left[\sum_{k=1}^{\infty}\nu_k(\pi)\tilde{ r}^k(2\pi,h)+\sum_{k=1}^{\infty}(-1)^k\nu_k(2\pi)\tilde{ r}^k(\pi,h)\right]_m=0.
 \end{equation}
If Lemma\ref{l2.1} holds and $\nu_1(\pi)=\nu_1(2\pi)=1$, so \eqref{e2.9} shows that
 \begin{equation}\label{e2.14}
 \begin{split}
    \left[\sum_{k=1}^{\infty}\nu_k(\pi)\tilde{ r}^k(2\pi,h)\right]_m=&\nu_m(\pi)+\nu_m(2\pi),\\
   \left[\sum_{k=1}^{\infty}(-1)^k\nu_k(2\pi)\tilde{ r}^k(\pi,h)\right]_m&=
    \Big[-\tilde{ r}(\pi,h)+(-1)^m\nu_m(2\pi)h^m\Big]_m\\&=-\nu_m(\pi)+(-1)^m\nu_m(2\pi).
    \end{split}
 \end{equation}
Furthermore, \eqref{e2.13} and \eqref{e2.14} yields that
\begin{equation}\label{e2.15}
   \Big[(-1)^m+1\Big]\nu_m(2\pi)=0.
\end{equation}
So if $m$ is an even number, \eqref{e2.15} yields that $\nu_m(2\pi)=0$, so Theorem \ref{t2.1} holds.
\end{proof}

\begin{theorem}\label{t2.2}
If $p$ is an odd number and $q$ is an even number, when $k>1$, the first subscript satisfying $\nu_k(2\pi)\neq 0$ in $\{\nu_k(2\pi)\}$ is even number.
\end{theorem}
\begin{proof}
Let $\theta=-\pi$ in \eqref{e2.1}, we have
\begin{equation}\label{e2.16}
    \tilde{r}(\pi,h)=\tilde{r}(-\pi,\tilde{r}(2\pi,h)).
\end{equation}
For any positive constant $m$,
\begin{equation}\label{e2.17}
   \left[ \tilde{r}(\pi,h)-\sum_{k=1}^{\infty}\nu_k(-\pi)\tilde{r}^k(2\pi,h)\right]_m=0.
\end{equation}
If Lemma\ref{l2.1} holds, \eqref{e2.9} and \eqref{e2.17} yield that
\begin{equation}\label{e2.18}
    \nu_m(\pi)-\nu_m(-\pi)-\nu_m(2\pi)=0.
\end{equation}
  Furthermore, let $\theta=2\pi$ in \eqref{e2.4}, we have
  \begin{equation}\label{e2.19}
    \tilde{r}(-\pi,h)+\tilde{r}(2\pi,-\tilde{r}(\pi,h))=0.
  \end{equation}
For any positive constant $m$,
\begin{equation}\label{e2.20}
     \left[\tilde{r}(-\pi,h)+\sum_{k=1}^{\infty}(-1)^k\nu_k(2\pi) \tilde{r}^k(\pi,h)\right]_m=0.
\end{equation}
When Lemma\ref{l2.1} holds, from \eqref{e2.20}, we get
\begin{equation}\label{e2.21}
   \Big[ \tilde{r}(-\pi,h)-\tilde{r}(\pi,h)+(-1)^m\nu_m(2\pi)\tilde{r}^m(\pi,h)\Big]_m=0.
\end{equation}
Namely
\begin{equation}\label{e2.22}
    \nu_m(-\pi)-\nu_m(\pi)+(-1)^m\nu_m(2\pi)=0.
\end{equation}
\eqref{e2.18} and \eqref{e2.22} show that
\begin{equation}\label{e2.23}
    \Big[(-1)^m-1\Big]\nu_m(2\pi)=0.
\end{equation}
If $m$ is an odd number, \eqref{e2.23} yields that $\nu_m(2\pi)=0$, so Theorem \ref{t2.2} holds.
\end{proof}

\begin{theorem}\label{t2.3}
If $p$ is an even number and $q$ is an odd number, when $k>1$, the first subscript satisfying $\nu_k(2\pi)\neq 0$ in $\{\nu_k(2\pi)\}$ is even number.
\end{theorem}
\begin{proof}
If $p$ is an even number and $q$ is an odd number, let $\theta=2\pi$ in \eqref{e2.6}, we have
\begin{equation}\label{e2.24}
    h+\tilde{r}(2\pi,-\tilde{r}(2\pi,h)=0.
\end{equation}
For any positive integer $m$,
\begin{equation}\label{e2.25}
   \left[ h+\sum_{k=1}^{\infty}(-1)^k\nu_k(2\pi)\tilde{r}^k(2\pi,h)\right]_m=0.
\end{equation}
 Suppose $\nu_k(2\pi)=0$ when $1<k<m$,then \eqref{e2.25} yields that
 \begin{equation}\label{e2.26}
   \Big[ -\tilde{r}(2\pi,h)+(-1)^m\nu_m(2\pi)\tilde{r}^m(2\pi,h)\Big]_m=0.
 \end{equation}
 We can get \eqref{e2.23} easily by \eqref{e2.26}, so Theorem\ref{t2.3} holds.
\end{proof}
Based on Theorem\ref{t2.1}, Theorem\ref{t2.2} and Theorem\ref{t2.3}, we have
\begin{definition}\label{d2.1}
If $p+q$ is even number, $\nu_{2m+1}(2\pi)$ is called to be the $m-th$ focal values of the origin of system \eqref{e1.10};\
If $p+q$ is odd number,$\nu_{2m}(2\pi)$ is called to be the $m-th$ focal values of the origin of system \eqref{e1.10}$m=1,2,\cdots$.
\end{definition}

\begin{definition}\label{d2.2}
If $p+q$ is even number, the first nonzero in $\{\nu_k(2\pi)\}$ is $\nu_{2m+1}(2\pi)$, then the origin is called the $m-$order weak (fine) focus of
\eqref{e1.10}; If $p+q$ is odd number, the first nonzero in $\{\nu_k(2\pi)\}$ is $\nu_{2m}(2\pi)$, hen the origin is called the $m-$order weak (fine) focus of
\eqref{e1.10}.
\end{definition}

The definition of algebraic equivalence was given in \cite{liu-2001}(see also \cite{liu-2014}).

\begin{definition}\label{d2.3}
 Suppose that $\{\lambda_m\}$ and $\{\tilde{\lambda}_m\}$ are polynomials of $(a_{kj})'s,\ (b_{kj})'s$ which are coefficient of functions of right hand of system \eqref{e1.8}, $\zeta_1^{(m)}$,\ $\zeta_2^{(m)}$,\ $\cdots$,\ $\zeta_{m-1}^{(m)}$ are polynomials of $(a_{kj})'s,\ (b_{kj})'s$.  If there exists a positive integer $m$  satisfy that
 \begin{equation}\label{e2.27}
    \lambda_m=\tilde{\lambda}_m+\left(\zeta_1^{(m)}\lambda_1+\zeta_2^{(m)}\lambda_2+\cdots+\zeta_{m-1}^{(m)}\lambda_{m-1}\right),
 \end{equation}
then $\lambda_m$ and $\tilde{\lambda}_m$ are algebraic equivalence, denote by $\lambda_m\sim\tilde{\lambda}_m$.\par
 If there exists a positive integer $m$ satisfy that $\lambda_m\sim\tilde{\lambda}_m$, then sequence of functions $\{\lambda_m\}$ and $\{\tilde{\lambda}_m\}$is called to be algebraic equivalence, denote by $\{\lambda_m\}\sim\{\tilde{\lambda}_m\}$.
\end{definition}
Definition\ref{d2.3} yields thatㄩ\par
1)\ Algebraic equivalence relation of sequence of functions is reflexive, symmetric, and transitive˙\par
2)\ If there exists a positive integer $m$ satisfy that$\lambda_m\sim\tilde{\lambda}_m$, when $\lambda_1=\lambda_2=\cdots=\lambda_{m-1}=0$,  $\lambda_m=\tilde{\lambda}_m$;\par
3)\ $\lambda_1\sim\tilde{\lambda}_1$ means that$\lambda_1=\tilde{\lambda}_1$.\par
From Theorem\ref{t2.1}, Theorem\ref{t2.2}, Theorem\ref{t2.3}, we can conclude that
\begin{theorem}\label{t2.4}
For system \eqref{e1.8}, if $p+q$ is an even number, then $\nu_{2m}(2\pi)\sim0$;
if $p+q$ is an odd number, then $\nu_{2m+1}(2\pi)\sim0$.\ $m=1,2,\cdots.$
\end{theorem}
Suppose that
\begin{equation}\label{e2.28}
    g(h)=h+\sum_{k=2}^{\infty}c_kh^k
\end{equation}
is a power series of $h$ with nonzero convergence radius. For sufficiently small $h$, it is more convenient to solve equation \eqref{e1.17} with initial conditions
\begin{equation}\label{e2.29}
    r|_{\theta=0}=g(h)
\end{equation}
than to solve equation $g(h)=h$ sometimes, (see the progress in \S3 for computing focal values of the origin of system \eqref{e3.1}).
The solution of\eqref{e1.17} can be written as a power series of $h$ with nonzero radius convergence when $|\theta|<4\pi$
\begin{equation}\label{e2.30}
    r=r^*(\theta,h)=\sum_{k=1}^{\infty}\nu^*_k(\theta)h^k,
\end{equation}
where
\begin{equation}\label{e2.31}
    \nu^*_1(0)=1,\ \ \nu^*_k(0)=c_k,\ k=2,3,\cdots.
\end{equation}
Furthermore,
\begin{equation}\label{e2.32}
    r=\tilde{r}(\theta,g(h))=\sum_{k=1}^{\infty}\nu_k(\theta)g(h)^k
\end{equation}
is another solution of \eqref{e1.17} with initial condition\eqref{e2.29}, so $r^*(\theta,h)=\tilde{r}(\theta,g(h))$ by uniqueness theory of solution. Namely
\begin{equation}\label{e2.33}
    \sum_{k=1}^{\infty}\nu^*_k(\theta)h^k=\sum_{k=1}^{\infty}\nu_k(\theta)g(h)^k
\end{equation}
\eqref{e2.33} shows that
\begin{theorem}\label{t2.5}
\begin{equation}\label{e2.34}
    \nu^*_k(2\pi)-\nu^*_k(0)\sim\nu_k(2\pi),\ k=2,3,\cdots.
\end{equation}
\end{theorem}

Next, we will consider the perturbed system of system \eqref{e1.10}
\begin{equation}\label{e2.35}
    \begin{split}
    &\frac{dx}{dt}=-py^{2p-1}+\sum_{m=2pq-q+1}^{\infty}\mathcal{X}_m(x,y,\bm\varepsilon)=\mathcal{X}(x,y,\bm\varepsilon),\\
    &\frac{dy}{dt}=qx^{2q-1}+\sum_{m=2pq-p+1}^{\infty}\mathcal{Y}_m(x,y,\bm\varepsilon)=\mathcal{Y}(x,y,\bm\varepsilon),
    \end{split}
\end{equation}
where
 \begin{equation}\label{e2.36}
    {\bm\varepsilon}=(\varepsilon_1,\varepsilon_2,\cdots,\varepsilon_n)
 \end{equation}
 is a small parameter and
     \begin{equation}\label{e2.37}
   \mathcal{X}_m(x,y,\bm\varepsilon)=\sum_{kp+jq=m}a_{kj}(\bm\varepsilon)x^ky^j,\ \   \mathcal{Y}_m(x,y,\bm\varepsilon)=\sum_{kp+jq=m}b_{kj}(\bm\varepsilon)x^ky^j
 \end{equation}
 are $m-$order homogeneous weight polynomial of $x,y$  with  weight $p,q$, $\mathcal{X}(x,y,\bm\varepsilon),\ \mathcal{Y}(x,y,\bm\varepsilon)$ are power series of $x, y,\bm\varepsilon$ with nonzero convergence radius. Denote the solution of system \eqref{e2.35} with initial condition $r|_{\theta=0}=h$ in generalized polar coordinate \eqref{e1.7} by
 \begin{equation}\label{e2.38}
    r=\tilde{r}(\theta,h,\bm\varepsilon)=\sum_{k=1}^{\infty}\nu_k(\theta,\bm\varepsilon)h^k.
\end{equation}
We have
\begin{theorem}\label{t2.6}
Suppose $p+q$ is an even number, if the origin of system \eqref{e2.35}is a $m-$th weak focus when $\bm\varepsilon=\mathbf{0}$,
and choosing a proper parameter $\bm\varepsilon$ in $|\bm\varepsilon|\ll1$, we have
\begin{equation}\label{e2.39}
    \begin{split}
    &0<|\nu_3(2\pi,\bm{\varepsilon})|\ll|\nu_5(2\pi,\bm{\varepsilon})|\ll\cdots|\nu_{2m+1}(2\pi,\bm{\varepsilon})|,\\
    &\nu_{2k-1}(2\pi,\bm{\varepsilon})\nu_{2k+1}(2\pi,\bm{\varepsilon})<0,\ k=2,3,\cdots,m,
    \end{split}
\end{equation}
then there exist $m-1$ limit cycles in the neighborhood of the origin of system \eqref{e2.35}.
\end{theorem}

\begin{theorem}\label{t2.7}
Suppose $p+q$ is an odd number, if the origin of system \eqref{e2.35} is a $m-$th weak focus when $\bm\varepsilon=\mathbf{0}$,
and choosing a proper parameter $\bm\varepsilon$ in $|\bm\varepsilon|\ll1$, we have
\begin{equation}\label{e2.40}
    \begin{split}
    &0<|\nu_2(2\pi,\bm{\varepsilon})|\ll|\nu_4(2\pi,\bm{\varepsilon})|\ll\cdots|\nu_{2m}(2\pi,\bm{\varepsilon})|,\\
    &\nu_{2k}(2\pi,\bm{\varepsilon})\nu_{2k+2}(2\pi,\bm{\varepsilon})<0,\ k=1,2,\cdots,m,
    \end{split}
\end{equation}
then there exist $m-1$ limit cycles in the neighborhood of the origin of system \eqref{e2.35}.
\end{theorem}
\begin{theorem}\label{t2.8}
Suppose $p+q$ is an odd number, if  the origin of system \eqref{e2.35} is a $m-$th weak focus when $\bm\varepsilon=\mathbf{0}$,
and the Jacobi matrix of $\nu_3(2\pi,\bm{\varepsilon}),\nu_5(2\pi,\bm{\varepsilon}),\cdots,\nu_{2m-1}(2\pi,\bm{\varepsilon})$ with respect to $\bm{\varepsilon}$ is
\begin{equation}\label{e2.41}
    J=\frac{\partial(\nu_3,\nu_5,\cdots,\nu_{2m-1})}{\partial(\varepsilon_1,\varepsilon_2,\cdots,\varepsilon_n)}
\end{equation}
There is a $m-1$ order determinant which do not equal to zero when $\bm\varepsilon=\mathbf{0}$, then choosing a proper parameter $\bm\varepsilon$ in $|\bm\varepsilon|\ll1$,there exist $m-1$ limit cycles in the neighborhood of the origin of system \eqref{e2.35}.
\end{theorem}

\begin{theorem}\label{t2.9}
Suppose $p+q$ is an odd number, if the origin of system \eqref{e2.27}is a $m-$th weak focus when $\bm\varepsilon=\mathbf{0}$,
and he Jacobi matrix of $\nu_2(2\pi,\bm{\varepsilon}),\nu_4(2\pi,\bm{\varepsilon}),\cdots,\nu_{2m-2}(2\pi,\bm{\varepsilon})$ with respect to $\bm{\varepsilon}$ is written as
\begin{equation}\label{e2.42}
    J=\frac{\partial(\nu_2,\nu_4,\cdots,\nu_{2m-2})}{\partial(\varepsilon_1,\varepsilon_2,\cdots,\varepsilon_n)}
\end{equation}
There is a $m-1$ order determinant which do not equal to zero when $\bm\varepsilon=\mathbf{0}$, then choosing a proper parameter $\bm\varepsilon$ in $|\bm\varepsilon|\ll1$,there exist $m-1$ limit cycles in the neighborhood of the origin of system \eqref{e2.35}.
\end{theorem}

\section{Center-focus determination and limit cycle bifurcation for $2:3$ homogeneous weight singular point}
 \setcounter{equation}{0}
Consider the following system
 \begin{equation}\label{e3.1}
    \begin{split}
    &\frac{dx}{dt}=-2y^3+(a_{22}x^2y^2+a_{50}x^5) =-2y^3+\mathcal{X}_{10}=\mathcal{X}(x,y),\\
    &\frac{dy}{dt}=3x^5+(b_{13} xy^3+ b_{41} x^4 y)=3x^5+\mathcal{Y}_{11}=\mathcal{Y}(x,y).
    \end{split}
 \end{equation}
 where
 \begin{equation}\label{e3.2}
\mathcal{X}_{10}=a_{22}x^2y^2+a_{50}x^5,\ \  \mathcal{Y}_{11}=b_{13} xy^3+b_{41} x^4 y.
 \end{equation}
Let
\begin{equation}\label{e3.3}
    x=r^2\cos\theta,\ \ y=r^3\sin\theta,
\end{equation}
we have
\begin{equation}\label{e3.4}
    \frac{dr}{d\theta}=r\ \frac{R_0(\theta)+R_1(\theta)r}
    {Q_0(\theta)+Q_1(\theta)r}.
\end{equation}
where
\begin{equation}\label{e3.5}
       \begin{split}
       &R_0(\theta)= \cos\theta\sin\theta \big(3 \cos^4\theta - 2 \sin^2\theta\big),\ \ \ \ Q_0(\theta)=6 \big(\cos^6\theta + \sin^4\theta\big),\\
       &R_1(\theta)=\cos\theta \big[\cos^3\theta \big(a_{50} \cos^2\theta + b_{41} \sin^2\theta)+\sin^2 \theta(a_{22} \cos^2\theta +
      b_{13} \sin^2\theta\big)\big],\\
      &Q_1(\theta)=-\cos^2\theta \sin\theta \big[\big(3 a_{50} -
      2 b_{41}) \cos^3\theta + \big(3 a_{22} - 2 b_{13}\big) \sin^2\theta\big].
        \end{split}
\end{equation}

 Denote the solution of equation \eqref{e3.4} with initial condition $r|_{\theta=0}=h$ by $r=\sum\limits_{k=1}^{\infty}\nu_k(\theta)h^k$, and
 \begin{equation}\label{e3.6}
    \nu_k(\theta)=\nu_1(\theta)u_k(\theta),\ k=2,3,\cdots,
 \end{equation}
 it is easy to get
 \begin{equation}\label{e3.7}
   \nu_1(\theta)=\frac{1}{\big(\cos^6\theta+\sin^4\theta\big)^{\frac{1}{12}}}\ ,
   \end{equation}
and
 \begin{equation}\label{e3.8}
 \begin{split}
  &u_2'(\theta)= \frac{\cos\theta \big(2 \cos^2\theta +
   3 \sin^2\theta\big)}{36 \big(\cos^6\theta + \sin^4\theta\big)^{\frac{25}{12}}}\\
   \times&\big(3 a_{50} \cos^9\theta +
   3 a_{22} \cos^6\theta \sin^2\theta +
   2 b_{41} \cos^3\theta \sin^4\theta + 2 b_{13} \sin^6\theta\big).
   \end{split}
   \end{equation}
\eqref{e3.8} yields that
   \begin{equation}\label{e3.9}
    \begin{split}
    u_2(\theta)&=\frac{-\cos^2\theta \sin\theta\left(2 b_{41} \cos^3\theta +
   5 b_{13} \sin^2\theta\right)}{60 \left(\cos^6\theta + \sin^4\theta\right)^{\frac{13}{12}}}\\
   +&\frac{1}{24}(2 a_{22} + 3 b_{13})f_2(\theta)+\frac{1}{60}(5 a_{50} + b_{41})g_2(\theta),
    \end{split}
\end{equation}
where
\begin{equation}\label{e3.10}
 \begin{split}
   & f_2(\theta)=\int_0^{\theta}\ \frac{\cos^7\varphi \sin^2\varphi \big(2 \cos^2\varphi + 3 \sin^2\varphi\big)}{\big(\cos^6\varphi + \sin^4\varphi\big)^{\frac{25}{12}}}\ d\varphi,\\
   &g_2(\theta)=\int_0^{\theta}\ \frac{\cos^{10}\varphi \big(2 \cos^2\varphi + 3 \sin^2\varphi\big)}{\big(\cos^6\varphi + \sin^4\varphi\big)^{\frac{25}{12}}}\ d\varphi.
    \end{split}
\end{equation}

\eqref{e3.6} and \eqref{e3.10} show the following proposition.
 \begin{proposition}\label{p3.1}
 The first focal value of system \eqref{e3.1} is
 \begin{equation}\label{e3.11}
    \nu_2(2\pi)=\frac{1}{60}(5 a_{50} + b_{41})\int_0^{2\pi}\ \frac{\cos^{10}\varphi (2 \cos^2\varphi + 3 \sin^2\varphi)}{\left(\cos^6\varphi + \sin^4\varphi\right)^{\frac{25}{12}}}\ d\varphi
 \end{equation}
 \end{proposition}

 Suppose the first focal value of origin of system \eqref{e3.1} is zero, namely,
 \begin{equation}\label{e3.12}
    5 a_{50} + b_{41}=0.
 \end{equation}
furthermore, we have
 \begin{equation}\label{e3.13}
    \begin{split}
    &u_3(\theta)=u_2^2(\theta)+\frac{1}{480} (3 a_{22} - 2 b_{13}) (2 a_{22} + 3 b_{13})f_3(\theta)\\
    +&\frac{\cos^4\theta}
    {50400 \big(\cos^6\theta + \sin^4\theta\big)^{\frac{13}{6}}}
    \ \Big[120 (2 a_{22} + 3 b_{13}) b_{41} \cos^9\theta \\-&
 7 (90 a_{22}^2 + 75 a_{22} b_{13} - 90 b_{13}^2 -
    52 b_{41}^2) \cos^6\theta \sin^2\theta\\-&
 40 (15 a_{22} - 23 b_{13}) b_{41} \cos^3\theta \sin^4\theta -
 350 (3 a_{22} - 2 b_{13}) b_{13} \sin^6\theta\Big]
    \end{split}
 \end{equation}
 where
\begin{equation}\label{e3.14}
   f_3(\theta)=\int_0^{\theta}\ \frac{\cos^{15}\varphi \sin\varphi \big(2 \cos^2\varphi + 3 \sin^2\varphi\big)}{\big(\cos^6\varphi + \sin^4\varphi\big)^{\frac{19}{6}}}\ d\varphi
\end{equation}
after more computation, we get
\begin{equation}\label{e3.15}
    u_4(\theta)=-u_2^3(\theta) + 2 u_2(\theta) u_3(\theta) + G_1w_2(\theta) +G_2+
  s_4f_4(\theta) + r_4g_4(\theta)
\end{equation}
where
\begin{equation}\label{e3.16}
\begin{split}
   &G_1 = \frac{\big(2 a_{22} +
     3 b_{13}\big) \cos^4\theta}{1209600 \big(\cos^6\theta + \sin^4\theta\big)^{
    \frac{13}{6}}}\Big[(120 \big(2 a_{22} +
        3 b_{13}\big) b_{41} \cos^9\theta\\
  & -
     7 \big(90 a_{22}^2 + 75 a_{22} b_{13} - 90 b_{13}^2 -
        52 b_{41}^2\big) \cos^6\theta \sin^2\theta\\
        & -
     40 \big(15 a_{22} - 23 b_{13}\big) b_{41} \cos^3\theta \sin^4\theta -
     350 \big(3 a_{22} - 2 b_{13}\big) b_{13} \sin^6\theta\Big],\\
       & G_2 = \frac{\sin^3\theta}{117936000 (\cos\theta^6 + \sin\theta^4)^{\frac{13}{4}}}\sum_{k=0}^5c_k\cos^{3(5-k)}\theta\sin^{2k}\theta.
     \end{split}
\end{equation}
\begin{equation}\label{e3.17}
\begin{split}
   &f_4(\theta)=\int_0^{\theta}\ \frac{\cos^{15}\varphi \sin\varphi \big(2 \cos^2\varphi + 3 \sin^2\varphi\big)}{\big(\cos^6\varphi + \sin^4\varphi\big)^{\frac{19}{6}}} f_2(\theta)\ d\varphi\\
   & g_4(\theta)=\int_0^{\theta}\ \frac{\cos^{20}\varphi \sin^2\varphi \big(2 \cos^2\varphi + 3 \sin^2\varphi\big)}{\big(\cos^6\varphi + \sin^4\varphi\big)^{\frac{17}{4}}}\ d\varphi,\\
   &s_4=\frac{1}{11520}(3 a_{22} - 2 b_{13}) (2 a_{22} + 3 b_{13})^2),\\
   &r_4=\frac{-13}{16800} (5 a_{22} - 3 b_{13}) (2 a_{22} + 3 b_{13}) b_{41}).
   \end{split}
\end{equation}
In \eqref{e3.16}
\begin{equation}\label{e3.18}
    \begin{split}
    &c_0=104 b_{41} (5550 a_{22}^2 + 4365 a_{22} b_{13} -
        5940 b_{13}^2 - 2548 b_{41}^2),\\
        &c_1= 273 (180 a_{22}^3 + 420 a_{22}^2 b_{13} + 45 a_{22} b_{13}^2 - 270 b_{13}^3 + 4104 a_{22} b_{41}^2 -1124 b_{13} b_{41}^2),
        \\
        &c_2= -54600 (10 a_{22}^2 - 33 a_{22} b_{13} + 19 b_{13}^2) b_{41}, \\
        &c_3=182 (3240 a_{22}^3 - 3690 a_{22}^2 b_{13} + 1935 a_{22} b_{13}^2 - 610 b_{13}^3 +
        1872 a_{22} b_{41}^2 +2808 b_{13} b_{41}^2),\\
       & c_4=0,\\
        &c_5=504 (2 a_{22} + 3 b_{13}) (180 a_{22}^2 - 225 a_{22} b_{13} + 70 b_{13}^2 +
        104 b_{41}^2).
        \end{split}
\end{equation}
 Above computation show that
 \begin{proposition}\label{p3.2}
when $\nu_2(2\pi)=0$, the second focal value at origin of system  \eqref{e3.1} is
 \begin{equation}\label{e3.19}
 \begin{split}
    &\nu_4(2\pi)=\frac{-13}{16800} (5 a_{22} - 3 b_{13}) (2 a_{22} + 3 b_{13}) b_{41}\\
    \times&\int_0^{2\pi}\ \frac{\cos^{20}\varphi \sin^2\varphi \big(2 \cos^2\varphi + 3 \sin^2\varphi\big)}{\big(\cos^6\varphi + \sin^4\varphi\big)^{\frac{17}{4}}}\ d\varphi.
    \end{split}
 \end{equation}
 \end{proposition}

The same method could be used to get $u_5(\theta),u_6(\theta)$,  then

 \begin{proposition}\label{p3.3}
 When $\nu_2(2\pi)=\nu_4(2\pi)=0$, the third focal value at origin of system  \eqref{e3.1} is
 \begin{equation}\label{e3.20}
   \nu_6(2\pi)=\frac{(2 a_{22} + 3 b_{13})^2 b_{41}^3}{412356420000}(575803A-11848200B),
 \end{equation}
where
 \begin{equation}\label{e3.21}
 \begin{split}
  &A=575803\int_0^{2\pi}\ \frac{\cos^{36}\varphi \big(2 \cos^2\varphi + 3 \sin^2\varphi\big)}{\big(\cos^6\varphi + \sin^4\varphi\big)^{\frac{77}{12}}}\ d\varphi\\
  & B=11848200\int_0^{2\pi}\ \frac{\cos^{28}\varphi\sin\varphi  \big(2 \cos^2\varphi + 3 \sin^2\varphi\big)}{\big(\cos^6\varphi + \sin^4\varphi\big)^{\frac{16}{3}}}f_2(\varphi)\ d\varphi
    \end{split}
 \end{equation}
 \end{proposition}
 With the aid of computer, we have
 \begin{equation}\label{e3.22}
    575803A-11848200B=814653.251446\cdots>0,
 \end{equation}
From Proposition \ref{p3.1}$\sim$Proposition \ref{p3.3}, we can conclude that
\begin{theorem}\label{t3.1}
The first three \eqref{e3.1} focal values at origin of system  \eqref{e3.1} are
\begin{equation}\label{e3.23}
    \begin{split}
   & V_2=5 a_{50} + b_{41},\\
   &V_4= -(5 a_{22} - 3 b_{13}) (2 a_{22} + 3 b_{13}) b_{41},\\
   &V_6=(2 a_{22} + 3 b_{13})^2 b_{41}^3.
    \end{split}
\end{equation}
\end{theorem}

\begin{theorem}\label{t3.2}
The origin of system \eqref{e3.1} is a center if and only if one of the following conditions holds:
\begin{equation}\label{e3.24}
    \begin{split}
    &C_1:\ \ a_{50}=-\frac{1}{5}b_{41},\ a_{22}=-\frac{3}{2}b_{13};\\
    &C_2:\ \ a_{50}=-\frac{1}{5}b_{41},\ b_{41}=0.
    \end{split}
\end{equation}
\end{theorem}
\begin{proof}
Theorem \ref{t3.1} yields that necessary condition holds. On the other hand, if condition $C_1$ in Theorem \ref{t3.2} holds, system \eqref{e3.1} is Hamilton; if condition $C_2$ in Theorem \ref{t3.2} holds, system \eqref{e3.1} is symmetric with $x$ axis, so the sufficient condition hold.
\end{proof}

Suppose the coefficients of functions at right hand of system  \eqref{e3.1} are
\begin{equation}\label{e3.25}
    b_{41}=1,\ a_{50}=-\frac{1}{5}(1-\varepsilon_1),\ a_{22}=\frac{1}{35}(5+7\varepsilon_2),\ b_{13}=\frac{5}{21},
\end{equation}
Theorem\ref{t3.1} shows that the first three \eqref{e3.1} focal values at origin of system  \eqref{e3.1} are
\begin{equation}\label{e3.26}
    V_2=\varepsilon_1,\ V_4|_{\varepsilon_1=0}=-\varepsilon_2+o(\varepsilon_2),\ V_6|_{\varepsilon_1=\varepsilon_2=0}=1,
\end{equation}
Furthermore, Theorem\ref{t2.5} yields that
\begin{theorem}\label{t3.3}
If\eqref{e3.25} holds, the origin of system \eqref{e3.1} is an unstable $3-$th weak focus when $0<\varepsilon_1\ll\varepsilon_2\ll1$, there exist two limit cycles in the neighborhood of origin of system \eqref{e3.1}.
\end{theorem}

Meanwhile, the origin of system \eqref{e3.1} is a high order singular point, the multiple degree of higher order singular point was defined in \cite{liu-1999} The origin of \eqref{e3.1} could be broken into some singular point with lower multiple degree, and limit cycles could be bifurcated from the new singular point. Similar problem have been investigated in \cite{liu-2009},\cite{liu-2015}. We consider a perturbed system of system \eqref{e3.1}
\begin{equation}\label{e3.27}
    \begin{split}
    &\frac{dx}{dt}=-\delta_0 \sigma^8 x - \sigma^6 y + \frac{1}{2} (5 a_{50} + b_{41} + 4 \delta_1 + 8 \delta_2) \sigma^2 x y^2 -
  2 y^3 + x^2 (a_{50} x^3 + a_{22} y^2),\\
  &\frac{dy}{dt}=\sigma^8 x - \delta_0 \sigma^8 y -\frac{1}{2}  (5 a_{50} + b_{41} - 4 \delta_1 + 8 \delta_2) \sigma^4 x^2 y +
 3 x^5 + x y (b_{41} x^3 + b_{13} y^2),
    \end{split}
\end{equation}
where $\delta_0,\delta_1,\delta_2$ and $\sigma$ are small parameters.\par
System \eqref{e3.27} is system \eqref{e3.1} when $\sigma=0$.\par
If $\sigma\neq0$,the origin of system \eqref{e3.27}is an elementary focus. In order to study its Hopf bifurcation,
system \eqref{e3.27} can be transformed into system
\begin{equation}\label{e3.29}
    \begin{split}
    &\frac{d\xi}{d\tau}=-\delta_0 \sigma \xi - \eta + \frac{1}{2} (5 a_{50} + b_{41} + 4 \delta_1 + 8 \delta_2) \sigma\xi \eta^2 -
  2 \eta^3 + \sigma\xi^2 (a_{50} \xi^3 + a_{22} \eta^2),\\
  &\frac{d\eta}{d\tau}= \xi - \delta_0 \sigma \eta -\frac{1}{2}  (5 a_{50} + b_{41} - 4 \delta_1 + 8 \delta_2) \sigma \xi^2 \eta +
 3 \xi^5 +  \sigma\xi \eta(b_{41} \xi^3 + b_{13} \eta^2).
    \end{split}
\end{equation}
 by transformation
\begin{equation}\label{e3.28}
    x=\sigma^2\xi,\ y=\sigma^3\eta,\ d\tau=\sigma^7dt.
\end{equation}
Furthermore, we can give that
\begin{theorem}\label{t3.4}
The divergence at origin of system \eqref{e3.29} is $\lambda_0=-2\delta_0\sigma$. When $\delta_0=0$, the first three \eqref{e3.29} focal values at origin of system  \eqref{e3.29} are
\begin{equation}\label{e3.30}
    \lambda_1=\delta_1\sigma,\ \ \lambda_2|_{\delta_1=0}=-\delta_2\sigma,\ \lambda_3|_{\delta_1=\delta_2=0}=\frac{47}{128}(5a_{50}+b_{41})\sigma.
\end{equation}
\end{theorem}

Theorem\ref{t3.3} and Theorem\ref{t3.4} show that
\begin{theorem}\label{t3.5}
If \eqref{e3.25} holds, when
\begin{equation}\label{e3.31}
    0<|\sigma|\ll1,\ 0<\delta_0\ll\delta_1\ll\delta_2\ll\varepsilon_1\ll\varepsilon_2\ll1,
\end{equation} there exist five limit cycles in the neighborhood of origin of system \eqref{e3.27}.
\end{theorem}

\end{document}